\documentclass[12pt]{article}

\usepackage{latexsym,amsmath,amssymb}
\usepackage{ascmac} 
\usepackage{amscd}
\usepackage{amsfonts}
\usepackage{mathrsfs}
\usepackage{bm}  
\usepackage{enumerate}  
\usepackage{hhline} 

\def\dis{\displaystyle}

\usepackage{amsthm}
\theoremstyle{definition}

\newtheorem{Theorem}{\bf Theorem}[section]
\newtheorem{Lemma}[Theorem]{\bf Lemma} 
\newtheorem{Proposition}[Theorem]{\bf Proposition} 
\newtheorem{Corollary}[Theorem]{\bf Corollary}
\newtheorem{Remark}[Theorem]{\bf Remark}
\newtheorem{Example}[Theorem]{\bf Example}
\newtheorem{Definition}[Theorem]{\bf Definition}

\newenvironment{thm}{\begin{Theorem}}{\end{Theorem}}
\newenvironment{lem}{\begin{Lemma}}{\end{Lemma}}
\newenvironment{prop}{\begin{Proposition}}{\end{Proposition}}

\newenvironment{rem}{\begin{Remark}}{\end{Remark}}

\newenvironment{dfn}{\begin{Definition}}{\end{Definition}}

\makeatletter
    
    \@addtoreset{equation}{section}
  \makeatother

\makeatletter
\def\tr{\mathop{\operator@font tr}\nolimits}  
\makeatother

\makeatletter
\def\dist{\mathop{\operator@font dist}\nolimits}  
\makeatother

\makeatletter
\def\div{\mathop{\operator@font div}\nolimits}  
\makeatother

\makeatletter
\def\exp{\mathop{\operator@font exp}\nolimits}  
\makeatother

\makeatletter
\def\essinf{\mathop{\operator@font {\it ess}.\inf}\nolimits}  
\makeatother

\makeatletter
\def\esssup{\mathop{\operator@font {\it ess.}\sup}\nolimits}  
\makeatother

\newcommand{\R}{\mathbb{R}}
\newcommand{\N}{\mathbb{N}}

\def\P{\mathcal{P}}
\def\O{\Omega}
\def\OO{\overline{\Omega}}
\def\L{\Lambda}
\def\a{\alpha}
\def\b{\beta}
\def\phi{\varphi}
\def\e{\varepsilon}
\def\d{\delta}
\def\s{\sigma}
\def\l{\lambda}
\def\o{\omega}
\def\p{\partial}
\def\fr{\frac}
\def\ti{\times}
\def\ra{\rangle}
\def\la{\langle}

\def\le{\left}
\def\ri{\right}
\def\fr{\frac}
\def\dis{\displaystyle}
\def\g{\gamma}
\def\ul{\underline}
\def\ol{\overline}
\def\Swiech{\'Swi\c{e}ch}

\begin{document}
\def\ns{\normalsize}

\title{On $L^p$-viscosity solutions of bilateral 
obstacle problems with unbounded ingredients}
\author{
\begin{tabular}{ccc}
\begin{tabular}{c}
Shigeaki Koike\footnote{Supported in part by Grant-in-Aid for Scientific Research (No. 16H06339, 
16H03948,  
16H03946
) of JSPS, e-mail: skoike@waseda.jp}\\
{\ns Deapartment of Applied Physics}\\
{\ns Waseda University}\\
{\ns Tokyo, 169-8555}\\
{\ns Japan}
\end{tabular}
&\& &\begin{tabular}{c}
Shota Tateyama\footnote{Supported by Grant-in-Aid for JSPS Research 
Fellow (No. 16J02399), and for Scientific Research (No. 16H06339), 
e-mail: shota.tateyama@gmail.com}\\
{\ns Department of Mathematics}\\
{\ns Waseda University}\\
{\ns Tokyo, 169-8555}\\
{\ns Japan}
\end{tabular}
\end{tabular}
}
\date{}

\pagenumbering{roman}
\maketitle

\pagenumbering{arabic}

\begin{abstract}
The global equi-continuity estimate on $L^p$-viscosity solutions of  bilateral obstacle problems with unbounded ingredients is established when  obstacles are merely continuous. 
The existence of $L^p$-viscosity solutions 
is established via an approximation of given data. 
The local H\"older continuity estimate on the first derivative of 
$L^p$-viscosity solutions is shown when the obstacles belong to $C^{1,\b}$, and $p>n$.  
\end{abstract}

\section{Introduction}
\label{sec:intro}

In this paper, we consider the following bilateral obstacle problem 
\begin{equation}\label{eq:1}
\min\{\max\{F(x, Du, D^2u)-f, u-\psi\}, u-\varphi\}=0\quad\mbox{in } \O,
\end{equation}
under the Dirichlet condition $u=g$ on $\p \O$,
 where 
  $\O\subset\R^n$ is  a bounded domain,  $F$ is at least a measurable function on $\O\ti\R^n\ti S^n$, and $f$, $\varphi$, $\psi$ and $g$ are given.   
We denote by $S^n$ the set of all $n\ti n$ real-valued 
symmetric matrices with the standard order, and  set 
$$
S^n_{\l,\L}:=\{ X\in S^n \ | \ \l I\leq X\leq \L I\}\quad\mbox{for }0<\l\leq\L.
$$

In contrast, unilateral obstacle problems are described by 
Bellman equations
\begin{equation}\label{UO1}
\max\{ F(x,Du,D^2u)-f, u-\psi\}=0\quad\mbox{in }\O,
\end{equation}
or 
\begin{equation}\label{UO2}
\min\{ F(x,Du,D^2u)-f, u-\phi\}=0\quad\mbox{in }\O.
\end{equation}

In \cite{LS67}, Lions-Stampacchia first introduced unilateral obstacle problems as an example of variational inequalities. 
Then, in \cite{BS, LS69}, regularity of solutions of obstacle problems was studied by Brezis-Stampacchia and Lewy-Stampacchia. 
Afterwards,  there  appeared numerous researches on unilateral obstacle problems 
when $F$  are partial differential operators of  divergence form. 
We only refer to  \cite{F, KS,PSU} and references therein 
for the existence and regularity of solutions of obstacle problems and applications.  

When $F$ is a linear second-order uniformly elliptic operator 
with smooth coefficients in \eqref{UO1} or \eqref{UO2}, 
as a crucial regularity result of solutions (i.e. $W^{2,\infty}(\O)$) of unilateral obstacle problems, we 
 refer to \cite{J}. 
We also refer to \cite{L} for $W^{2,\infty}_{loc}(\O)$ regularity of solutions of \eqref{UO1} 
when $F$ is given by the maximum of a finite number of 
linear second-order uniformly elliptic operators  
with smooth coefficients.  
  
We also note that unilateral  obstacle problems arise in stochastic optimal stopping time problems. We refer to \cite{FS, To} and references therein  for this issue.

Going back to bilateral obstacle problems, we  refer to \cite{MR} and \cite{DMV}, respectively, for a nice review and a pioneering regularity result. 
As an application, we also refer to \cite{DY}.

We note that equation \eqref{eq:1} is  formally equivalent to the following problem:
$$
\le\{
\begin{array}{rl}
F(x,Du,D^2u)\leq f(x) &\mbox{in }\{ x\in\O \ | \ u(x)>\phi (x)\},\\
F(x,Du,D^2u)\geq f(x)&\mbox{in }\{ x\in\O \ | \ u(x)<\psi (x)\},\\
\phi\leq u\leq\psi&\mbox{in }\O.
\end{array}\ri.$$
Furthermore, we notice that 
\eqref{eq:1} can be regarded as the following  Isaacs equation
$$
\min_{\a\in [0,1]}\max_{\b\in [0,1]}\{ F_{\a,\b}(x,u,Du,D^2u)-f_{\a,\b}\} =0\quad\mbox{in }\O,
$$
where for two parameters $\a,\b\in [0,1]$,
$$
F_{\a,\b}(x,r,\xi,X)=\a\b F(x,\xi,X)+\a(1-\b)r +(1-\a)r$$
and
$$
f_{\a,\b}(x)=\a\b f(x)+\a(1-\b)\psi (x)+(1-\a)\phi (x),$$
because of the fact that for $A,B,C\in\R$, 
$$
A\wedge (B\vee C)=\min_{\a\in [0,1]}\max_{\b\in [0,1]}
\{ \a\b A+\a (1-\b)B+(1-\a)C\}.
$$
Here and later, we use the notations: for $a,b\in\R$, 
$$
a\vee b:=\max\{ a,b\}, \ a\wedge b:=\min\{ a,b\}, a^+:=a\vee 0\mbox{ and }a^-:=(-a)\vee 0.$$

On the other hand, we have few results when $F$ has 
non-divergence structure even for unilateral obstacle problems. 
Duque in \cite{D} recently showed  
  interior H\"older estimates on viscosity solutions of bilateral obstacle problems for fully nonlinear uniformly elliptic operators with no 
 variable coefficients, no first derivative terms and constant inhomogeneous terms but 
 only assuming that the obstacles are H\"older continuous; 
$$
\le\{\begin{array}{l}
F(x,\xi,X)=F(X)\mbox{ for }(x,\xi,X)\in \O\ti\R^n\ti S^n,\\
f\equiv C, \\
\phi,\psi\in C^\a(\O)\mbox{ for some }\a\in (0,1).
\end{array}\ri.
$$
Assuming  the above hypotheses, 
in \cite{D}, we obtain  the existence of viscosity solutions of \eqref{eq:1} 
under the Dirichlet condition, and  
interior H\"older estimates on the first derivative of 
viscosity solutions of \eqref{eq:1} when obstacles are in $C^{1,\b}$ 
for $\b\in (0,1)$. 
The results associated with parabolic problems are also shown in \cite{D}. 
We refer to \cite{LP, LPS} for very recent related topics, and to \cite{CLM} for a different approach 
via Tug-of-War games. 

Although a  clever use of the weak Harnack inequality was adapted 
to show those estimates in \cite{D}, 
in order to extend the results to more general $F$ and $f$, it seems difficult to 
establish the estimates near the free boundary and near $\p\O$. 

Our aim in this paper is to extend results in \cite{D} when $F$ is a fully nonlinear uniformly elliptic operator. 
More precisely, under more general hypotheses than those in \cite{D}, 
we show the equi-continuity of $L^p$-viscosity solutions of \eqref{eq:1} in $\OO$, 
the existence of $L^p$-viscosity solutions of \eqref{eq:1}, 
and 
their local H\"older continuity of derivatives 
under  additional assumptions.

For the corresponding results of parabolic obstacle problems,  we cannot use 
the argument in the proof of H\"older estimates on the derivative of $L^p$-viscosity solutions 
because the domain, where the infimum is taken, differs from that of the $L^{\e_0}$ 
(quasi)-norm in the weak Harnack inequality, 
which arises in Proposition \ref{prop:WH} for the elliptic case. 
The second author finds a new argument to avoid this difficulty. 
We refer to \cite{T} for the parabolic version of this paper.

For any $p>0$ and  $u:\O\to\R$, we denote the quasi-norm: 
$$
\| u\|_{L^p(\O)}=\le(\int_\O |u(x)|^p dx\ri)^{\fr1p}.$$
We note that $\|\cdot\|_{L^p(\O)}$ satisfies 
\begin{equation}\label{eq:tri}
\| u+v\|_{L^p(\O)}\leq C_p\le(\| u\|_{L^p(\O)}+\|v\|_{L^p(\O)}\ri) \quad \mbox{for some }C_p\geq 1,
\end{equation}
where  $C_p=1$ provided $p\geq 1$. 

This paper is organized as follows: 
In Section 2, we recall the definition of $L^p$-viscosity solutions, basic properties, and exhibit main results. 
Section 3 is devoted to the weak Harnack inequality  
both in $\O'\Subset\O$ and near $\p\O$, which 
yields the global equi-continuity of $L^p$-viscosity solutions. 
In Section 4, 
we establish the existence of $L^p$-viscosity solutions of \eqref{eq:1} when 
the obstacles are only continuous under appropriate hypotheses. 
We obtain H\"older estimates on the first derivative of 
$L^p$-viscosity solutions in Section 5. 

\begin{center}
Acknowledgements
\end{center}
The authors thank the referees for their careful reading,  and several valuable comments, which help us to improve the original manuscript.


\section{Preliminaries and main results}

For any $x\in \R^n$ and  $r>0$, we set
$$B_r:=\{y\in \R^n: |y|<r\},\quad \mbox{and}\quad B_r(x):=x+B_r.$$
For any measurable set $A\subset\R^n$, we denote by $|A|$ the Lebesgue measure of $A$. 

We recall the definition of $L^p$-viscosity solutions of general elliptic 
partial differential equations (PDE for short) from \cite{CCKS}:
\begin{equation}\label{GElliptic}
G(x,u,Du,D^2u)=0\quad\mbox{in }\O,
\end{equation}
where $G:\O\ti\R\ti\R^n\ti S^n\to \R$ is measurable.

\begin{dfn}
We call $u\in C(\O)$ an $L^p$-viscosity subsolution (resp., supersolution) 
of \eqref{GElliptic} 
if whenever $u-\eta$  attains its local maximum (resp., 
minimum) at $x_0\in \O$ for $\eta\in W^{2,p}_{loc}(\O)$, it follows that 
$$
\lim_{r\to 0} \ ess\inf_{B_r(x_0)} G(x,u(x),D\eta (x),D^2\eta (x))\leq 0$$
$$
\le(\mbox{resp., }\lim_{r\to 0} \ ess\sup_{B_r(x_0)} G(x,u(x),D\eta (x),D^2\eta (x))\geq 0\ri) .$$
We also call $u$ an $L^p$-viscosity solution of \eqref{GElliptic} if it is both an $L^p$-viscosity sub- and supersolution of \eqref{GElliptic}. 
\end{dfn}

\begin{rem}
We will call $C$-viscosity subsolutions (resp., supersolutions, solutions) if we replace $W^{2,p}_{loc}(\O)$ by $C^2(\O)$ in the above when given $G$ is continuous. 
We refer to \cite{CIL} for the theory of $C$-viscosity solutions. 
\end{rem}

In order to present our main results, we shall  prepare some notations and hypotheses.  
Throughout this paper, under the hypothesis 
\begin{equation}\label{Apq}
p_0<p\leq q, \quad q>n,
\end{equation}
where $p_0\in [\fr{n}{2}, n)$ is the constant in \cite{Es}, 
we suppose 
\begin{equation}\label{Af}
f\in L^p(\O).
\end{equation}

Concerning $F$, we suppose  that there exist constants $0<\l\leq\L$, and 
\begin{equation}\label{Amu}
\mu\in L^q(\O)
\end{equation}
such that 
\begin{equation}\label{UE}
\P^-_{\l, \L}
(X-Y)-\mu(x)|\xi-\eta|\leq F(x, \xi, X)-F(x,\eta,Y)\leq\P^+_{\l, \L}
(X-Y)+\mu(x)|\xi-\eta|
\end{equation}
for $x\in \O, \xi,\eta\in\R^n$, $X,Y\in S^n$, where $\P^\pm_{\l, \L}:S^n\to\R$ are defined by 
$$
\P^+_{\l, \L}(X):=\max\{-\mathrm{Tr}(AX): A\in S^n_{\l, \L} \}\quad\mbox{and}\quad 
\P^-_{\l, \L}(X):=-\P^+_{\l,\L}(-X)
$$
for $X\in S^n$. 
Since we fix $0<\l\leq\L$ in this paper,  we shall write $\P^\pm:=\P^\pm_{\l, \L}$ for simplicity. 
We also suppose that   
\begin{equation}\label{Azero}
 F(x,0,O)=0\quad\mbox{for }x\in \O .
\end{equation}
We notice that  \eqref{UE} and \eqref{Azero} yield
$$
\mu \geq 0\quad\mbox{in }\O.
$$

For obstacles $\phi,\psi$ and the Dirichlet datum 
$g$, as compatibility conditions, we suppose 
\begin{equation}\label{Aobstacles}
\phi\leq \psi\quad\mbox{in }\O,\quad\mbox{and}\quad\phi\leq g\leq \psi\quad\mbox{on }\p\O.
\end{equation}


\subsection{Basic properties}

We first give a direct consequence from the definition, which will be 
often used.

\begin{prop}\label{rem:1}
Assume \eqref{Apq}, \eqref{Af}, \eqref{Amu}, \eqref{UE}, \eqref{Azero} and \eqref{Aobstacles}. 
Let $u\in C(\O)$ be an $L^p$-viscosity subsolution 
(resp., supersolution) of $(\ref{eq:1})$. 
Assume that $\g\in \R$ satisfies $\g\geq\varphi$ (resp. $ \g\leq\psi$) in an open set $\O_0\subset\O$. Then, $u\vee \g$ (resp. $u\wedge \g$) is an $L^p$-viscosity subsolution (resp., supersolution) of 
$$
\P^-(D^2u)-\mu|Du|-f^+=0 \quad
\le(\mbox{resp., } \P^+(D^2u)+\mu|Du|+f^-=0\ri)
\quad \mbox{in } \O_0.
$$
\end{prop}
\begin{proof}
We only prove the assertion for subsolutions. 

For $\xi\in W^{2,p}_{loc}(\O_0)$, we suppose that 
$(u\vee \g) -\xi$ attains its local maximum at $x_0\in\O_0$. 

If we assume $u(x_0)>\g$, then $u-\xi$ attains its local maximum at $x_0\in\O_0$, and $u>\phi$ near $x_0$. 
Hence, by the definition, we have
$$
\lim_{r\to 0}ess\inf_{B_r(x_0)} \max\{ F(x,D\xi(x),D^2\xi(x))-f(x),(u-\psi)(x)\}\leq 0,$$
which yields the conclusion by \eqref{UE}. 

When $u(x_0)\leq \g$, it is enough to show that any constant is an $L^p$-viscosity subsolution of 
\begin{equation}\label{PucciEq}
\P^-(D^2u)=0\quad\mbox{in }\O .
\end{equation}
In fact, by noting that any constant is a $C$-viscosity subsolution of \eqref{PucciEq}, in view of Proposition 2.9 in \cite{CCKS}, 
it is also an $L^p$-viscosity subsolution of \eqref{PucciEq}. 
\end{proof}

We shall recall the scaled version of the weak Harnack inequality 
and the H\"older continuity in \cite{KS3}. 
Modifying the result in \cite{KS3} by an argument of the compactness, 
 we state the next proposition as simple as possible for later use. 
See \cite{KS3} for the original version. 
Here and later, we use the notation
$$
\a_0:=2-\fr{n}{p\wedge n}\in (0,1].
$$


\begin{prop}
\label{prop:WH} (cf. Theorem 4.5, 4.7, Corollary 4.8 in \cite{KS3}) 
Assume \eqref{Apq}, \eqref{Af} and \eqref{Amu}. 
There exist $\e_0>0$ and $C_0>0$ such that 
if $v\in C(B_{2r})$ is a nonnegative $L^p$-viscosity supersolution of 
\begin{equation}\label{eq:P+}
\P^+(D^2v)+\mu|Dv|-f=0\quad \mbox{in }B_{2 r},
\end{equation}
then it follows that 
$$\| v\|_{L^{\e_0}(B_r)}\leq C_0r^{\fr{n}{\e_0}}\le(\inf_{B_r} v+r^{\a_0}\|f\|_{L^{p\wedge n}(B_{2r})}\ri).$$
Here, $\e_0$ and $C_0$ depend on $n,\fr\L\l, p,q$ and $\|\mu\|_{L^q(B_{2 r})}$. 
\end{prop}

In Section 5, we will use the following local maximum principle. 

\begin{prop}\label{prop:LMP}(cf. Theorem 3.1 in \cite{KS4}) 
Under hypotheses \eqref{Apq}, \eqref{Af}, \eqref{Amu}, for any $\e>0$, 
there exists $C_1=C_1\le(n,\l,\L,\e,p,q,\|\mu\|_{L^q(B_{2r})}\ri)>0$ such that 
if $u\in C(B_{2r})$ is a nonnegative $L^p$-viscosity subsolution  of 
$$
\P^-(D^2u)-\mu |Du|-f=0\quad \mbox{in }B_{2r},$$
then it follows that  
$$
\dis\sup_{B_{\fr r2}}u\leq C_1\le(r^{-\fr n\e}\| u\|_{L^\e(B_{r})}+ r^{\a_0}\| f^+\|_{L^{p\wedge n}(B_{2r})}\ri)
.$$
\end{prop}

Although  it is mentioned in Theorem 6.2 of \cite{KS3} that 
Proposition \ref{prop:WH} implies the H\"older continuity of $L^p$-viscosity solutions of 
\begin{equation}\label{eq:no_obstacle}
F(x,Du,D^2u)-f=0\quad\mbox{in }\O,
\end{equation}
to show a key idea of this paper, we recall how to derive  H\"older estimates on $L^p$-viscosity solutions of \eqref{eq:no_obstacle}.


\begin{prop}\label{prop:Holder}(cf. Theorem 6.2 in \cite{KS3})
Assume \eqref{Apq}, \eqref{Af}, \eqref{Amu},  \eqref{UE} and \eqref{Azero}. 
Let $\O:=B_{2R}$ for $R\in (0,1]$. 
Then, there exist  constants $K_1>0$ and $\hat\a\in (0,\a_0]$ such that 
if  $u\in C(B_{2R})$ is an $L^p$-viscosity 
solution of \eqref{eq:no_obstacle}, then it follows that
$$
|u(x)-u(y)|\leq K_1\le(\fr{|x-y|}{R}\ri)^{\hat\a}\le(\| u\|_{L^\infty (B_{2R})}+R^{\a_0}\| f\|_{L^{p\wedge n}(B_{2R})}\ri) \quad \mbox{for }x,y\in B_{R}
.$$
\end{prop}

\begin{proof}
Fix $x\in B_{R}$. 
For $0<s\leq R$, 
we  set
$$
M_s:=\sup_{B_s(x)} u,\quad\mbox{and}\quad m_s:=\inf_{B_s(x)} u.
$$
Now, for $0<r\leq \fr R2$, setting
$$
U:=u-m_{2r}\geq 0,\quad\mbox{and}\quad V:=M_{2r}-u\geq 0\quad\mbox{in }
B_{2r}(x),$$
we immediately see that $U$ and $V$ are $L^p$-viscosity supersolutions of \eqref{eq:P+} with $f$ replaced by $-f^-$ and $-f^+$, respectively. 
Hence, in view of Proposition \ref{prop:WH}, we have 
$$
\| U\|_{L^{\e_0}(B_r(x))}\leq C_0r^{\fr{n}{\e_0}}\le(\inf_{B_r(x)} U+r^{\a_0}
\| f^-\|_{L^{p\wedge n}(B_{2r}(x))}\ri),
$$
$$
\| V\|_{L^{\e_0}(B_r(x))}\leq C_0r^{\fr{n}{\e_0}}\le(\inf_{B_r(x)} V+r^{\a_0}
\| f^+\|_{L^{p\wedge n}(B_{2r}(x))}\ri).
$$
Therefore, in view of Proposition \ref{prop:WH}, we can find $C_0'>1$ such that 
$$
\begin{array}{rcl}
M_{2r}-m_{2r}&=&|B_r|^{-\fr{1}{\e_0}}\| M_{2r}-m_{2r}\|_{L^{\e_0}(B_r(x))}
\\
&\leq &|B_r|^{-\fr{1}{\e_0}}C_{\e_0}\le(
\| V\|_{L^{\e_0}(B_r(x))}+\| U\|_{L^{\e_0}(B_r(x))}\ri)\\
&\leq&C_0'\le(M_{2r}-M_r+m_r-m_{2r}+r^{\a_0}\| f\|_{L^{p\wedge n}(B_{2R})}\ri),
\end{array}
$$
where $C_{\e_0}\geq 1$ is from \eqref{eq:tri}. 
Thus, there exists  $\theta_0\in (0,1)$ such that
$$
\o (r)\leq \theta_0 \o (2r)+ r^{\a_0}\| f\|_{L^{p\wedge n}(B_{2R})},
$$
where $\o (r)=M_r-m_r$. 
Hence, the standard argument (e.g. Lemma 8. 23 in \cite{GT}) implies that 
$$
|u(x)-u(y)|\leq K_1\le(\fr{|x-y|}{R}\ri)^{\hat\a}\le( \| u\|_{L^\infty (B_{2R})}+R^{\a_0}
\| f\|_{L^{p\wedge n}(B_{2R})}\ri)
$$
for some $K_1>0$ and $\hat\a\in (0,\a_0]$. 
\end{proof}


\begin{rem}
One of key ideas of this paper is a different choice of $M_s$ and $m_s$ in the above for the proof of Lemma \ref{lem:L}. 
\end{rem}

When $p>n$ as in \eqref{Apq2} in Section 2. 2, 
we recall  the following regularity result for fully nonlinear PDE.


\begin{prop}\label{prop:C1beta}(\cite{Caf, CafCab,S}) 
Let $\O=B_2$. 
Under \eqref{Apq}, \eqref{Af},  
there exist $\hat\b\in (0,1)$ and $K_2>0$ such that if $u\in C(B_{2})$ is an $L^p$-viscosity subsolution and 
$L^p$-viscosity supersolution, respectively, of 
$$
\P^-(D^2u)=0\quad \mbox{and}\quad \P^+(D^2u)=0\quad\mbox{in }B_2,$$
then it follows that
$$
\| u\|_{C^{1,\hat\b}(\ol B_1)}\leq K_2\| u\|_{L^\infty(B_2)}.$$
\end{prop}

We finally give a reasonable property of $L^p$-viscosity solutions of \eqref{eq:1}, which will be often used without mentioning it. 
We present a proof for the reader's convenience though it seems 
standard. 


\begin{prop}\label{prop:phi_u_psi}
Under \eqref{Apq}, \eqref{Af}, \eqref{Amu}, \eqref{UE}, \eqref{Azero} and \eqref{Aobstacles}, we assume $\phi,\psi\in C(\O)$. 
If $u\in C(\O)$ is an $L^p$-viscosity subsolution (resp., supersolution)  of \eqref{eq:1}, then it follows that
$$u\leq\psi\quad (\mbox{resp., }u\geq \phi)\quad\mbox{in }\O.$$
\end{prop}

\begin{proof}
We give  a proof only for $L^p$-viscosity subsolutions since the other case can be shown similarly. 
Assume that $(u-\psi)(x_0)=:\theta >0$ for $x_0\in\O$, then we will 
have a contradiction. 
For simplicity, we may suppose $x_0=0\in\O$ by translation. 

For $\e>0$, we let $x_\e\in \ol \O$ be such that 
$\max\{ u(x)-\fr 1{2\e}|x|^2 \ | \ x\in\ol\O\}=u(x_\e)-\fr 1{2\e}|x_\e|^2$. 
Since it is easy to see that $\dis\lim_{\e\to 0}x_\e= 0$, we may suppose $x_\e\in\O$. 
Moreover,  we may suppose $u\geq \psi +\fr \theta 2$ 
in $B_r\Subset \O$ for some $r>0$. 
Thus, by the first inequality in \eqref{Aobstacles}, 
we have 
\begin{equation}\label{2prop:1}
u\geq \psi+\fr \theta2\geq \phi +\fr \theta2\quad\mbox{in }B_r.
\end{equation}
However, from the definition, we have
$$
0\geq \lim_{s\to 0}ess\inf_{B_s(x_\e)}\min\le\{ \max\le\{
F\le(x,\fr x\e,\fr 1\e I\ri)-f(x),(u-\psi)(x)\ri\},(u-\phi)(x)\ri\}
,$$
which yields
$$
0\geq \min\{ (u-\psi)(x_\e),(u-\phi)(x_\e)\}.$$
This contradicts to \eqref{2prop:1}. 
\end{proof}

\subsection{Main results}

For obstacles, we at least  assume that 
\begin{equation}\label{Aregob}
\phi, \psi\in C(\ol\O). 
\end{equation}

In order to obtain the estimate near $\p\O$, 
we suppose the following condition on the shape of $\O$, which was  introduced in \cite{BNV}.  
\begin{equation}\label{Aboundary}
\le\{\begin{array}{c}
\mbox{There exist }R_0>0 \mbox{ and }\Theta_0>0
\mbox{ such that}\\
|B_r(x)\setminus \O|\geq \Theta_0r^n\mbox{ for }(x,r)\in \p\O\ti (0,R_0).
\end{array}
\ri.
\end{equation}
We will also suppose 
\begin{equation}\label{AregDirichlet}
g\in C(\p\O). 
\end{equation}

We call a function $\o:[0,\infty)\to [0,\infty)$ a modulus of continuity  if $\o$ is nondecreasing and continuous in $[0,\infty)$ such that $\o(0)=0$.


Our first result is the global equi-continuity estimate on $L^p$-viscosity solutions. 

\begin{thm}\label{corEst}
Assume \eqref{Apq}, \eqref{Af}, \eqref{Amu}, \eqref{UE}, \eqref{Azero}, \eqref{Aobstacles},  \eqref{Aregob}, 
 \eqref{Aboundary} and \eqref{AregDirichlet}. 
Then, there exists a modulus of continuity $\o_0$ such that 
if $u\in C(\ol\O)$ is an $L^p$-viscosity solution of \eqref{eq:1} satisfying
\begin{equation}\label{BC}
u=g\quad\mbox{on }\p\O ,
\end{equation}  
then it follows that  
$$
|u(x)-u(y)|\leq \o_0(|x-y|)\quad \mbox{for }x,y\in \ol\O.
$$
If we moreover assume that
$$
\phi,\psi\in C^{\a_1}(\ol\O),\quad\mbox{and}\quad g\in C^{\a_1}(\p\O)\quad\mbox{for }\a_1\in (0,1),$$
then there exist $\a_2\in (0,\a_0\wedge \a_1]$ and $C>0$, independent of $u$, 
 such that
$$
|u(x)-u(y)|\leq C|x-y|^{\a_2}\quad\mbox{for }x,y\in\ol\O.$$
\end{thm}

Thanks to Theorem \ref{corEst}, we establish the following existence result.


\begin{thm}\label{thm:Exist}
Under  \eqref{Apq}, \eqref{Af}, \eqref{Amu}, \eqref{UE}, \eqref{Azero}, \eqref{Aobstacles},  \eqref{Aregob}, 
 and \eqref{AregDirichlet}, 
 we assume the uniform exterior cone condition on $\O$. 
 Then,  
there exists an $L^p$-viscosity solution $u\in C(\ol\O)$ of \eqref{eq:1} 
satisfying \eqref{BC}. 
\end{thm}

For further regularity results, assuming 
\begin{equation}\label{Apq2}
q\ge p>n,
\end{equation}
we define $\b_0\in (0,1)$ by 
$$
\b_0:=1-\fr np.$$
To show $C^{1,\b}$ estimates, we will see in Section 5 that 
it is necessary to suppose 
that 
\begin{equation}\label{Aregob2}
\phi,\psi\in C^{1,\b_1}(\O)\quad \mbox{for some }\b_1\in (0,1). 
\end{equation}
We will use the constant
$\b_2$ defined by
$$
\b_2:=\b_0\wedge \b_1\in (0,1). 
$$

We also suppose that obstacles do not coincide in $\O$; 
\begin{equation}\label{Aobstacles2}
\mbox{there is } r_0>0\mbox{ such that }\psi-\phi\geq r_0\mbox{ in }\O .
\end{equation}

In order to state the next theorem, we prepare some notations. 
For small $r>0$, we introduce subdomains of $\O$:
$$\O_r :=\{x\in\O \ | \ \mbox{dist}(x,\p\O)>r\} .$$

For $u\in C(\O)$ such that $\phi\leq u\leq \psi$ in $\O$, 
 we set 
$$C^-[u]:=\le\{x\in \O \ | \ u(x)=\varphi(x)\}, \ C^+[u]:=\{x\in \O \ | \ 
u(x)=\psi(x)\ri\},$$
$$C^\pm[u]:=C^-[u]\cup C^+[u]\subset\O,$$
and  the non-coincidence set
$$
N[u]:=\O\setminus C^\pm [u]=\{ x\in\O \ | \ \phi (x)<u(x)<\psi(x)\}. $$
For small $r>0$, we define subdomains of $N[u]$
$$ 
N_r[u]:=\{ x\in \O_r \ | \ \mbox{dist}(x,C^\pm[u])>r\}.$$


For $F$ in \eqref{eq:1}, we use the following notation:
$$
\theta ( x,y):=\sup_{X\in S^n}\fr{|F(x,0,X)-F(y,0,X)|}{1+\| X\|}\quad\mbox{for }x,y\in\O.$$


\begin{thm}\label{thm:C1Holder1}
Assume \eqref{Apq2}, \eqref{Af}, \eqref{Amu}, \eqref{UE}, \eqref{Azero}, \eqref{Aregob2} and \eqref{Aobstacles2}. 
For each small $\e>0$, there exist $C>0$ and $\d_0>0$ such that 
if $u\in C(\O)$ is an $L^p$-viscosity solution of \eqref{eq:1}, and 
if 
 \begin{equation}\label{Aconti}
\fr 1r\| \theta (y,\cdot )\|_{L^n(B_r(y))}\leq \d_0 \quad\mbox{for }r\in (0,\e)\mbox{ and }y\in N_{\e}[u], 
\end{equation} 
then 
it follows that 
$$
|Du(x)-Du(y)|\leq C|x-y|^{\b_3}\quad\mbox{for }x,y\in \O_\e,$$
where 
$$
\b_3:=\le\{\begin{array}{ll}
\b_2&\mbox{provided }\b_2<\hat\b,\\
\mbox{any }\b\in (0,\hat\b)&\mbox{provided }\b_2\geq \hat\b.
\end{array}\ri.
$$
\end{thm}


\section{Global equi-continuity estimates}

In what follows, assuming \eqref{Aregob}, 
we denote by $\s_0$ the modulus of continuity of $\varphi$ and $ \psi$ in $\ol\O$;
$$
\s_0(r):=\max\{ |\phi (x)-\phi(y)|\vee |\psi(x)-\psi(y)| \ | \ |x-y|\leq r, \ 
x,y\in \ol\O\}.
$$
 

\subsection{Local estimates}

We first show the local equi-continuity estimate on $L^p$-viscosity solutions 
of \eqref{eq:1}. 


\begin{lem}\label{lem:L}
Assume \eqref{Apq}, \eqref{Af}, \eqref{Amu}, \eqref{UE}, \eqref{Azero},  \eqref{Aobstacles} 
 and \eqref{Aregob}. 
For any 
$\O'\Subset \O$, 
there exists a modulus of continuity $\o_0$ such that if $u\in C(\O)$ is an $L^p$-viscosity solution of \eqref{eq:1}, 
then it follows that 
$$
|u(x)-u(y)|\leq \o_0(|x-y|)\quad \mbox{for }x,y\in\ol\O'.
$$
 \end{lem}

\begin{proof}
Let $r\in(0, \fr{\d_0}{2})$, where $\d_0:=\mbox{dist}(\O',\p\O)$, and $x_0\in\ol \O'$. 
We may suppose $x_0=0$ as before. 
Setting  $\s_0:=\s_0(2r)$,  we define 
$$
u_+:=  u\vee (\phi(0)+\s_0)\quad\mbox{and}\quad u_-:= u\wedge (\psi(0)-\s_0).
$$
By noting $\phi(0)+\s_0 \geq  \phi$ and $\psi\geq \psi(0)-\s_0$ in $B_{2r}$, 
Proposition \ref{rem:1} shows that  
$u_+$ and $u_-$ are, respectively, an $L^p$-viscosity subsolution 
and supersolution of
$$
\P^-(D^2u)-\mu |Du|-f^+=0\quad\mbox{and}\quad \P^+(D^2u)+\mu|Du|+f^-=0\quad\mbox{in }B_{2r}.
$$

Now, for $s\in (0,\fr{\d_0}{2})$, setting 
$$ M_s:=\underset{B_{s}}{\sup} \ u_+, \quad\mbox{and}\quad 
m_s:=\underset{B_{s}}{\inf} \ u_-,$$
we define
$$
 U:=M_{2r}-u_+
,\quad\mbox{and}\quad  V:=u_--m_{2r}
$$
for $r\in (0,\fr{\d_0}{2})$.  

It is easy to 
see that $U$ and $V$ are, respectively, 
nonnegative $L^p$-viscosity 
supersolutions of 
$$
\P^+(D^2u)+ \mu|Du|+ f^\pm=0\quad\mbox{in }B_{2r}. 
$$
Hence, by Proposition \ref{prop:WH}, 
we have 
\begin{equation}\label{wH1}
\| U\|_{L^{\e_0}(B_r)}\leq Cr^{\fr{n}{\e_0}}\le(\inf_{B_r} \ U+r^{\a_0}\| f^+\|_{L^{p\wedge n}(\O)}\ri)
,
\end{equation}
and
\begin{equation}\label{wH2}
\| V\|_{L^{\e_0}(B_r)}\leq Cr^{\fr{n}{\e_0}}\le(\inf_{B_r} \ V+r^{\a_0}\| f^-\|_{L^{p\wedge n}(\O)}\ri)
.
\end{equation}
Here and later, $C>0$ denotes the various constant depending only on known quantities. 
Since $
M_{2r}-m_{2r}=U+(u_+-u)+(u-u_-)+V\leq 
U+4\s_0+V$ by Proposition \ref{prop:phi_u_psi}, we have
$$
M_{2r}-m_{2r}\leq Cr^{-\fr{n}{\e_0}}\le(
\| U\|_{L^{\e_0}(B_r)}+\s_0 r^{\fr{n}{\e_0}}+\| V\|_{L^{\e_0}(B_r)}\ri)
.
$$
Combining this with \eqref{wH1} and \eqref{wH2}, 
we find $\theta_0\in (0,1)$ such that
$$
M_r-m_r\leq \theta_0\le( M_{2r}-m_{2r}\ri) +r^{\a_0}
\| f\|_{L^{p\wedge n}(\O)}+\s_0(2r). 
$$
We note here that 
$$
u(x)-u(y)\leq u_+(x)- u_-(y)\quad\mbox{for }x,y\in B_{2r}.$$
Therefore, as for Proposition \ref{prop:Holder} with Lemma 8.23 in \cite{GT}, 
it is standard to find a modulus of continuity $\o_0$ in the conclusion. 
\end{proof}

 
\begin{rem}\label{rem:Holder}
As noted in Section 2.2, if we suppose $\phi,\psi\in C^{\a_1}(\O)$ for $\a_1\in (0,1)$, then we can show $u\in C^{\a_2}(\ol\O')$ for some 
$\a_2\in (0,\a_0\wedge \a_1]$ because we can choose $\s_0(r)=Cr^{\a_1}$ for some $C>0$ in the above. 
\end{rem}


\subsection{Equi-continuity near  $\p \O$}

To state equi-continuity near $\p\O$, we shall use the following notion: for small $\e>0$, 
$$\O_\e :=\{x\in\O \ | \ \mbox{dist}(x,\p\O)>\e\} .$$


\begin{lem}\label{lem:B}
Assume \eqref{Apq}, \eqref{Af}, \eqref{Amu}, \eqref{UE}, \eqref{Azero},  \eqref{Aobstacles},  \eqref{Aregob}, \eqref{Aboundary} and 
\eqref{AregDirichlet}. 
For small $\e>0$, there exists a modulus of continuity $\o_0$ such that 
if  an $L^p$-viscosity solution $u\in C(\ol\O)$ of \eqref{eq:1}
 satisfies \eqref{BC}, 
then it follows that 
$$
|u(x)-u(y)|\leq \o_0(|x-y|)\quad \mbox{for }x,y\in \ol\O\setminus \O_{\e}.
$$
\end{lem}

\begin{proof}
Let $r\in(0, \fr{\e}{2})$ and $x_0\in\p \O$. 
We may suppose $x_0=0\in\p\O$. 
As in the proof of Lemma \ref{lem:L}, we set 
$$
u_+:=  u\vee (\phi(0)+\s_0)\quad\mbox{and}\quad u_-:=  u\wedge (\psi(0)-\s_0),
$$
where $\s_0:=\s_0(2r)$. 
In view of Proposition \ref{rem:1} again, we see that 
$u_+$ and $u_-$ are, respectively, an $L^p$-viscosity subsolution 
and supersolution of
$$
\P^-(D^2u)-\mu |Du|-f^+=0\quad\mbox{and}\quad \P^+(D^2u)+\mu|Du|+f^-=0\quad\mbox{in }B_{2r}\cap\O.
$$
Now, as in \cite{GT, KS3} for instance, setting 
$$M_s:=\underset{B_{s}\cap \O}{\sup}u_+, \quad\mbox{and}\quad 
m_s:=\underset{B_{s}\cap \O}{\inf}u_-,$$
we define
$$
U:=\le\{\begin{array}{ll}
(M_{2r}-u_+)\wedge c_+&\mbox{in }B_{2r}\cap \O,\\
c_+&\mbox{in }B_{2r}\setminus \O,
\end{array}\ri.
$$
and
$$
V:=\le\{\begin{array}{ll}
(u_--m_{2r})\wedge c_-&\mbox{in }B_{2r}\cap \O,\\
c_-&\mbox{in }B_{2r}\setminus \O,
\end{array}\ri.$$
where nonnegative constants $c_\pm$ are given by
$$
c_+:=M_{2r}-\sup_{B_{2r}\cap \p\O}u_+\quad\mbox{and}\quad 
c_-:=\inf_{B_{2r}\cap \p\O} u_--m_{2r}.
$$
Hence, it is easy to 
see that $U$ and $V$ are 
nonnegative $L^p$-viscosity 
supersolutions of 
$$
\P^+(D^2u)+\hat \mu|Du|+|\hat f|=0\quad\mbox{in }B_{2r},
$$
where $\hat \mu$ and $\hat f$ are zero extensions of $\mu$ and $f$  
 outside of $\O$, respectively. 
Hence, by Proposition \ref{prop:WH}, 
we have 
$$
 \Theta_0^{\fr{1}{\e_0}} \le( M_{2r}-\sup_{B_{2r}\cap \p\O}  u_+
 \ri) \leq C\le( M_{2r}-\sup_{B_r\cap \O} u_++r^{\a_0}\| f\|_{L^{p\wedge n}
(\O)}\ri)  
$$
and
$$
\Theta_0^{\fr{1}{\e_0}}\le( \inf_{B_{2r}\cap \p\O} u_- -m_{2r}\ri)\leq 
C\le(\inf_{B_r\cap \O} u_--m_{2r}+r^{\a_0}\| f\|_{L^{p\wedge n}(\O)}\ri)
. 
$$
As in the proof of Lemma \ref{lem:L}, these inequalities imply that there is $\theta_0\in (0,1)$ such that
$$
M_r-m_r\leq \theta_0\le( M_{2r}-m_{2r}\ri) +2r^{\a_0}
\| f\|_{L^{p\wedge n}(\O)}+\s_0(2r)+\o_g(2r),
$$
where $\o_g(r):=\sup\{ |g(x)-g(y)| \ | \ |x-y|\leq r, \ x,y\in \p\O\}$. 
Therefore, noting that 
$u(x)-u(y)\leq u_+(x)-u_-(y)$ for $x,y\in B_{2r}\cap\O$, 
 as before, we can find a modulus of continuity $\o_0$ in the assertion. 
\end{proof}


\begin{rem}
As in Remark \ref{rem:Holder}, if we suppose 
$$
\phi,\psi\in C^{\a_1}(\ol\O\setminus\O_{2\e}),\quad\mbox{and}\quad g\in C^{\a_1}(\p\O)
\quad\mbox{for }\a_1\in (0,1),$$
then $u\in C^{\a_2}(\ol\O\setminus\O_\e)$ holds for some $\a_2\in (0,\a_0\wedge \a_1]$. 
\end{rem}

\begin{proof}[Proof of theorem \ref{corEst}] 
In view of Lemmas \ref{lem:L} and \ref{lem:B}, we immediately obtain the assertion. 
\end{proof}


\section{Existence results}

In this section, we present an existence result of 
$L^p$-viscosity solutions of \eqref{eq:1} under suitable conditions 
when obstacles are merely continuous.  

Using the standard mollifier by $\rho\in C^\infty_0(\R^n)$, we introduce smooth approximations of $f$ and $F$ by
$$
f_\e:=f*\rho_\e, \ \mu_\e:=\mu*\rho_\e\quad\mbox{and}\quad 
F_\e(x,\xi,X):=\int_{\R^n}\rho_\e (x-y)F(y,\xi,X)dy$$
for $(x,\xi,X)\in\R^n\ti\R^n\ti S^n$, where $\rho_\e(x)=\e^{-n}\rho(x/\e)$. 
Here and later, we use the same notion $f$ and $F$ for their zero extension outside of $\O$. 
Under \eqref{Af}, \eqref{Amu},  \eqref{UE} and \eqref{Azero}, 
it is easy to observe that for $(x,\xi,X)\in\R^n\ti \R^n\ti S^n$, 
\begin{equation}\label{ApproxProp}
\le\{\begin{array}{ll}
(i)&\P^-(X)-\mu_\e (x)|\xi|\leq F_\e(x,\xi,X)\leq \P^+(X)+\mu_\e (x)|\xi|,\\
(ii)&\| f_\e\|_{L^p(\R^n)}\leq \| f\|_{L^p(\O)},\\
(iii)&\|\mu_\e\|_{L^q(\R^n)}\leq \| \mu\|_{L^q(\O)}.
\end{array}
\ri.
\end{equation}

Furthermore, we shall suppose that $\phi$ and $\psi$ are defined in a neighborhood of $\ol\O$ with the same modulus of continuity. 
More precisely,  there is $\e_1>0$ such that
\begin{equation}\label{Aphipsi}
\max\{ |\phi(x)-\phi(y)|, |\psi(x)-\psi(y)|\}\leq \s_0(|x-y|)
\quad\mbox{for }x,y\in N_{\e_1},
\end{equation}
where $N_{\e_1}:=\{ x\in \R^n \ | \ \mbox{dist}(x,\O)<\e_1\}$ is a neighborhood of $\ol\O$. 
Under \eqref{Aphipsi}, we define $\phi_\e$ and $\psi_\e$ as follows:
$$
\phi_\e:=\phi*\rho_\e-\s_0(\e),\quad
\psi_\e:=\psi*\rho_\e +\s_0(\e).$$
It is easy to see that for $\e\in (0,\e_1)$, 
$$
\phi_\e\leq g\leq\psi_\e\quad\mbox{on }\p\O,$$
and
$$
\max\{|\phi_\e (x)-\phi_\e(y)|, |\psi_\e(x)-\psi_\e(y)|\}\leq \s_0(|x-y|)\quad \mbox{for }x,y\in\ol\O .
$$


We shall consider approximate equations:
\begin{equation}\label{Approximate}
F_\e (x,Du,D^2u)+\fr 1\d (u-\psi_\e)^+-\fr 1\d (\phi_\e-u)^+=f_\e\quad\mbox{in }\O.
\end{equation}

In order to apply an existence result in \cite{CKLS}, we shall suppose the uniform exterior cone condition on $\p\O$ in \cite{Miller}, which is stronger than \eqref{Aboundary}. 


\begin{prop}\label{prop:Exist}
Under \eqref{Apq}, \eqref{Af}, \eqref{Amu}, \eqref{UE}, \eqref{Azero}, \eqref{Aobstacles}, \eqref{Aregob},   \eqref{AregDirichlet} and \eqref{Aphipsi}, we assume the uniform exterior cone condition on $\O$. 
Then, there exists a $C$-viscosity solution $u^\d_\e\in C(\ol\O)$ of \eqref{Approximate} satisfying \eqref{BC}. 
\end{prop}

We first show an existence result for \eqref{eq:1} when $\phi$, $\psi$ and $F$ are smooth. 


\begin{thm}\label{thm:Approx}(cf. Theorem 1.1 in \cite{CKLS}) 
Under the same hypotheses in Proposition \ref{prop:Exist}, let 
$u^\d_\e\in C(\ol\O)$ be $C$-viscosity solutions of \eqref{Approximate} 
satisfying \eqref{BC}. 
For each small $\e>0$, there exist $\d_\e>0$ and $\hat C_\e>0$ such that 
\begin{equation}\label{EstBeta}
0\leq \fr 1\d(u^\d_\e-\psi_\e)^++\fr 1\d (\phi_\e-u^\d_\e)^+\leq \hat C_\e\quad \mbox{in }\ol\O\quad\mbox{for }\d\in (0,\d_\e). 
\end{equation}
Furthermore, there exist a subsequence $\{\d_k\}_{k=1}^\infty$ and 
$u_\e\in C(\ol\O)$ such that $\d_k\to 0$ as $k\to\infty$, \eqref{BC} holds for $u_\e$, 
\begin{equation}\label{Converge}
u^{\d_k}_\e\to u_\e\quad \mbox{uniformly in }\ol\O, \mbox{ as }k\to\infty,
\end{equation}
and $u_\e$ is a (unique) $C$-viscosity solution of 
\begin{equation}\label{ApproximateEq}
\min\{\max\{ F_\e(x,Du,D^2u)-f_\e,u-\psi_\e\}, u-\phi_\e\}=0\quad\mbox{in }\O.
\end{equation} 
\end{thm}

\begin{proof}
To show the estimate on $\fr 1\d (u^\d_\e-\psi_\e)^+$, independent of $\d>0$, we let $x_0\in \ol\O$ be such that
$$
\max_{\ol\O}\fr 1\d(u^\d_\e-\psi_\e)^+=\fr 1\d (u^\d_\e-\psi_\e)^+(x_0)>0.$$
Thus, we see that $x_0\in\O$, and $u^\d_\e-\psi_\e$ attains its maximum at $x_0\in\O$. 
Hence, the definition implies
$$
0\leq\fr 1\d(u^\d_\e-\psi_\e)^+\leq f_\e-F_\e (x_0,D\psi_\e,D^2\psi_\e)\quad\mbox{at }x_0$$
because of $\fr 1\d(\phi_\e-u^\d_\e)^+(x_0)=0$. 

Following the same argument, we obtain the estimate on $\fr 1\d(\phi_\e-u^\d_\e)^+$. 
Thus, we conclude the first assertion \eqref{EstBeta}. 
We then obtain the $L^\infty$ bound of $u^\d_\e$ independent of $\d\in (0,1)$ for each $\e\in (0,1)$. 

By regarding the penalty term as the right hand side with $L^\infty$-estimates, independent of $\d>0$, it is standard to show the equi-continuity and unifrom boundedness of $\{ u^\d_\e\}_{\d>0}$ 
for each $\e>0$. 
Therefore, by Ascoli-Arzela theorem, we can find a subsequence $\{ u^{\d_k}_\e\}_{k=1}^\infty$ and $u_\e\in C(\ol\O)$ satisfying \eqref{Converge}. 

We shall show that $u_\e$ is a $C$-viscosity subsolution of \eqref{ApproximateEq}  by contradiction. 
Thus, we suppose that $u-\eta$ attains its  local strict maximum at $x_0\in\O$ for $\eta\in C^2(\O)$, and 
\begin{equation}\label{Verify}
\min\{\max\{ F_\e(x_0,D\eta,D^2\eta)-f_\e,u_\e-\psi_\e\},u_\e-\phi_\e\}\geq 2\theta\quad\mbox{at }x_0
\end{equation}
for some $\theta>0$. 
By the uniform convergence, we may suppose that $u^{\d_k}_\e-\eta$ attains its local maximum at $x_{\d_k}\in \O$, where 
$x_{\d_k}\to x_0$ as $k\to \infty$. 
In what follows, we shall write $\d$ for $\d_k$. 

By \eqref{Verify}, since we may suppose
$$
(u^\d_\e-\phi_\e)(x_\d)\geq \theta,$$
we have $\fr 1\d(\phi_\e-u^\d_\e)^+=0$ at $x_\d$. 
Hence, sending $k\to \infty$ in \eqref{Approximate} with $\d=\d_k$, we have 
$$
F_\e(x_0,D\eta,D^2\eta)\leq f_\e\quad\mbox{at }x_0, 
$$
which together with \eqref{Verify} yields
$$
(u^\d_\e-\psi_\e)(x_\d)\geq \theta
$$
for small $\d>0$. 
However, 
this together with \eqref{EstBeta}  yields a contradiction for large  $k\geq 1$. 
\end{proof}

Now, we shall show our existence result.


\begin{proof}[Proof of Theorem \ref{thm:Exist}]
Let $u_\e\in C(\ol\O)$ be $C$-viscosity solutions of \eqref{ApproximateEq} satisfying \eqref{BC} constructed in Theorem \ref{thm:Approx}. 
Since $F_\e$ and $f_\e$ are continuous, it is known to see that $u_\e$ 
is an $L^p$-viscosity solution of \eqref{ApproximateEq}. 
We refer to \cite{CKLS} for instance. 
Furthermore, recalling \eqref{ApproxProp}, thanks to Theorem \ref{corEst}, 
we find a modulus of continuity $\o_0$ such that
$$
|u_\e(x)-u_\e(y)|\leq \o_0(|x-y|)\quad\mbox{for }x,y\in\OO.$$
Hence, by Proposition \ref{prop:phi_u_psi}, we can find a subsequence 
$\e_k>0$ and $u\in C(\ol\O)$ such that $\e_k\to 0$, as $k\to \infty$, and 
$u_{\e_k}$ converges to $u$ uniformly in $\ol\O$. 
For simplicity, we shall write $\e$ for $\e_k$. 

It remains to show that $u$ is an $L^p$-viscosity solution of \eqref{eq:1}. 
To this end, we suppose that for some $\eta\in W^{2,p}_{loc}(\O)$, 
$u-\eta$ attains its  local strict maximum 
at $x_0\in\O$, and 
$$
\min\{\max\{ F(x,D\eta,D^2\eta)-f,u-\psi\},u-\phi\}\geq 2\theta\quad\mbox{a.e. in } B_{2r}(x_0)\Subset\O$$
for some $\theta ,r>0$. 
For the sake of simplicity, we shall suppose $x_0=0\in\O$. 
Since we may suppose that for small $\e>0$, 
$$
u_\e-\phi_\e\geq\theta \quad\mbox{in }B_r,
$$
it is enough to consider the case when $u_\e$ is an $L^p$-viscosity subsolution of 
\begin{equation}\label{ExistenceSub}
\max\{ F_\e (x,Du,D^2u)-f_\e, u-\psi_\e\}=0\quad\mbox{in }B_r.
\end{equation}
Thus,  Proposition \ref{prop:phi_u_psi} implies 
$$
u\leq \psi\quad\mbox{in }B_r.$$
Hence,  $\eta\in W^{2,p}(B_r)$ satisfies 
\begin{equation}\label{ExistenceContra}
F(x,D\eta,D^2\eta)\geq f+\theta\quad\mbox{a.e. in }B_r.
\end{equation}

On the other hand, following the argument in the proof of Theorem 4.1 in \cite{CKLS}, since $u_\e$ is an $L^p$-viscosity subsolution of 
\eqref{ExistenceSub} together with the uniform convergence of $u_\e$ to 
$u$, we obtain that $u$ is an $L^p$-viscosity subsolution of 
$$
F(x,Du,D^2u)-f=0\quad\mbox{in }B_r, $$
which contradicts   \eqref{ExistenceContra}. 
We only notice that 
 $\mu D\eta\in L^p(B_r)$ holds true  since $q>n$, and $\eta\in W^{2,p}(B_r)$ for $q\geq p$ though  $\mu$ may not be in $ L^\infty$ in \eqref{UE}.  
\end{proof}


\section{Local H\"older continuity of derivatives}

It is well-known that we cannot expect solutions of obstacle problems to be 
in $C^2$ even when obstacles are in $C^2$. 
Furthermore, since $\phi_0(x):=-|x|^{1+\b}+1$ for $x\in [-1,1]$ with $\b\in (0,1)$ is a $C$-viscosity solution of 
$$
\min\{ -u'',u-\phi_0\}=0\quad\mbox{in }(-1,1)$$
under the Dirichlet condition $g\equiv 0$, 
we cannot expect solutions to be in $W^{2,\infty}_{loc}$ when 
obstacles only belong to $C^{1,\b}$. 
Notice that since there is no $C^2$ function which touches $\phi_0$ from below at the origin, we do not have to check the definition of $C$-viscosity supersolutions at $0$.


\subsection{Estimates in the non-coincidence set $N[u]$}

We first note that $L^p$-viscosity solutions $u\in C(\O)$ of \eqref{eq:1} 
are also $L^p$-viscosity solutions of
\begin{equation}\label{eq:F=f}
F(x,Du,D^2u)-f(x)=0\quad\mbox{in }N[u].
\end{equation}

For any compact $K\Subset N[u]$, where $u\in C(\O)$ is an 
 $L^p$-viscosity solution of \eqref{eq:1}, 
we show that $Du\in C^{\b}(K)$ for some $\b\in (0,\hat\b)$, 
where $\hat\b\in (0,1)$ is the constant in Proposition \ref{prop:C1beta}. 


\begin{prop}\label{prop:C1Holder}(cf. Theorem 2.1 in \cite{S})
Assume 
\eqref{Apq2}, \eqref{Af}, \eqref{Amu}, \eqref{UE}, \eqref{Azero}, \eqref{Aobstacles} and \eqref{Aregob}. 
Then,  there are $\b \in (0,\hat \b)$, 
$C>0$, 
$\d_0>0$ and  $r_1>0$, depending on $n,\fr\L\l, p,q
$,  
such that if $u\in C(\O)$ is an $L^p$-viscosity solution of \eqref{eq:1}, and 
if \eqref{Aconti} holds for $\e=r_1$, 
then  $u\in C^{1,\b}(N_{r_1} [u])$. 
More precisely, if $B_r(x)\subset N_{r_1} [u]$, then  it follows that
$$
|Du(y)-Du(z)|\leq \fr{C|y-z|^{\b}}{r^{1+\b}}\le(\| u\|_{L^\infty (\O)}+r^{1+\b_0}\| f\|_{L^p(\O)}\ri)\quad\mbox{for } y,z\in B_{\fr r2}(x).$$
\end{prop}


\begin{rem}
For further estimates on $L^p$-viscosity solutions of \eqref{eq:F=f} under some additional assumptions, 
we refer to Theorem B. 1 in \cite{CCKS}. 
When $F$ is given by the maximum of  finite uniformly elliptic operators with smooth coefficients, 
 we also refer to \cite{E} for $C^{2,\a}$-estimates. 
However, when we have $\mu\in L^q$ and  $f\in L^p$, we could only expect $u$ to be in $W^{2,p}$.

 We moreover refer to \cite{Te} for some precise equi-continuity estimates on $C$-viscosity solutions 
of \eqref{eq:F=f} when $\mu\equiv 0$ in \eqref{UE} (i.e. $F$ is independent of $\xi\in\R^n$). 
\end{rem}

Before going to the proof of Proposition \ref{prop:C1Holder}, 
we first show a lemma corresponding to Lemma 2.3 in \cite{S}. 
See also \cite{Caf,CafCab}.

For a modulus of continuity $\rho$ and a constant $K>0$, we introduce  
$$
C(\rho,K;\ol \O):=\le\{\xi\in C(\ol \O) \ \le| \ 
\begin{array}{c}
|\xi (x)-\xi(y)|\leq \rho (|x-y|)\mbox{ for}\\
x,y\in \p \O,\mbox{ and }\|\xi\|_{L^\infty (\O)}\leq 
  K
\end{array}\ri.\ri\}.
$$


\begin{lem}\label{prop:approximate}(cf. Lemma 2.3 in  \cite{S}) 
Assume \eqref{Apq2},  \eqref{Amu}, \eqref{UE} and \eqref{Azero} 
with $\O=B_2$. 
For given $G:B_2\ti \R^n\ti S^n\to \R$, we let
$$
g^*(x):=\sup\{ |G(x,\xi,X)| \ | \ \xi\in\R^n, X\in S^n\}.
$$ 
For  a modulus of continuity $\rho$, and for constants  $K,\e>0$ and $p'\in (n,p)$, 
there exists $\d_0=\d_0 (\e,p',n,\fr\L\l, p,q,\rho,K)\in (0,1)$ such that if 
\begin{equation}\label{eq:smallness}
\| g^*\|_{L^n(B_2)}\vee \| \mu\|_{L^{p'}(B_2)}\vee \sup_{x\in B_2} \| \theta (x,\cdot)\|_{L^n(B_2)}\leq \d_0,
\end{equation}
then for any two $L^p$-viscosity solutions $v$ and $\xi\in C(\rho,K;\ol B_2)$ of 
$$
F(x,Du,D^2u)+G(x,Du,D^2u)=0\quad\mbox{in }B_2
$$
and
$$
F(0,0,D^2u)=0\quad\mbox{in }B_2,
$$
respectively, satisfying   $(v-\xi)|_{\p B_2}=0$, 
it follows that 
$$
\| v-\xi\|_{L^\infty (B_2)}\leq \e.$$
\end{lem}

\begin{rem}
We notice that $\|\mu\|_{L^{p'}(B_2)}\leq\d_0$ in \eqref{eq:smallness} for 
$p'\in (n,p)$ because we do not know if the equi-continuity of $v_k$ holds true 
in the proof below when $p'=n$. 
\end{rem}


\begin{proof}
We argue by contradiction. 
Thus, suppose that there are $\hat\e>0$,  $v_k,\xi_k\in  C(\rho,K;\ol B_2)$, 
$\mu_k\in L^q(B_2)$, $G_k:B_2\ti \R^n\ti S^n\to \R$, and $F_k:B_2\ti\R^n\ti S^n\to \R$ satisfying \eqref{UE} with $\mu_k$; 
$$
\P^-(X-Y)-\mu_k(x)|\xi-\eta|\leq F_k(x,\xi,X)-F_k(x,\eta,Y)\leq \P^+(X-Y)+\mu_k (x)|\xi-\eta|$$
for $x\in B_2$, $\xi,\eta\in \R^n$, $X,Y\in S^n$ (for $k\in \N$) such that 
\begin{equation}\label{approximate:hypo}
\| g^*_k\|_{L^n(B_2)}\vee \| \mu_k\|_{L^{p'}(B_2)}\vee \sup_{x\in B_2}
\| \theta_k(x,\cdot )\|_{L^n(B_2)}\leq \fr 1k,
\end{equation}
where 
$g^*_k(x)=\sup\{ |G_k(x,\xi,X)| \ | \ (\xi,X)\in\R^n\ti S^n\}$, and 
$$
\theta_k(x,y)=\sup_{X\in S^n}\fr{|F_k(x,0,X)-F_k(y,0,X)|}{1+\| X\|},$$
$v_k$ and $\xi_k$ are, respectively, $L^p$-viscosity solutions 
of
$$
F_k(x,Dv_k,D^2v_k)+G_k(x,Dv_k,D^2v_k)=0\quad\mbox{and} \quad F_k(0,0,D^2\xi_k)=0\quad\mbox{in }B_2,
$$
which satisfy that $
(v_k-\xi_k)|_{\p B_2}=0$, and
\begin{equation}\label{approximate:conclusion}
\| \xi_k-v_k\|_{L^\infty (B_2)}\geq \hat\e.
\end{equation}

Since we may suppose that there are $v,\xi\in C(\rho,K;\ol B_2)$ such that 
$v_k$ and $\xi_k$ converges  to $v$ and $\xi$ uniformly in $\ol B_2$, respectively,  and 
$v=\xi$ on $\p B_2$. 
Because the mapping $X\in S^n\to F_k(0,0,X)$ is bounded by \eqref{UE}, 
we may suppose $F_k(0,0,X)$ converges to $F_\infty (X)$, which satisfies
$$
F_\infty (O)=0,\quad\mbox{and}\quad \P^-(X-Y)\leq F_\infty (X)-F_\infty (Y)\leq\P^+(X-Y).
$$
We also notice that by \eqref{UE} and our assumption \eqref{approximate:hypo},  
$$
\lim_{k\to\infty}\sup\le\{ \| F_k(\cdot ,\xi,X)-F_\infty (X)\|_{L^n(B_2)} \ \le| \ 
|\xi|\leq R, \| X\|\leq R\ri.\ri\}=0$$
holds for each $R>0$. 
Hence, since $F_\infty$ is continuous, in view of Lemma 1. 7 in \cite{S}, 
we verify that 
$v$ and $\xi$ are $L^n$-viscosity (thus, $C$-viscosity) solutions of 
$$
F_\infty (D^2u)=0\quad\mbox{in }B_2.$$
Therefore, the comparison principle implies that $
v=\xi$ in $\ol B_2$, which contradicts \eqref{approximate:conclusion}. 
\end{proof}

Although our proof of Proposition \ref{prop:C1Holder} follows by 
the same argument as in \cite{Caf, S}, 
we give a  proof  because we need some modification.


\begin{proof}[\ul{Proof of Proposition \ref{prop:C1Holder}}] 
Recalling $\b_0:=1-\fr np\in (0,1)$ and $\hat \b\in (0,1)$ from Proposition \ref{prop:C1beta}, 
we fix $\b\in (0,1)$ and $p'>n$ such that
$$
0<\b<\b_0\wedge \hat \b\quad\mbox{and} \quad p':=\fr{p+n}{2}\in (n,p).$$

For small $s\in (0,1)$, which will be fixed later,  setting
$$
\e:=K_2s^{1+\hat \b},$$
we choose $\d_0=\d_0(\e,p',n,\fr\L\l,p,q,\rho)\in (0,1)$ in Lemma  \ref{prop:approximate}, where the modulus of continuity $\rho$ is given by
$$
\rho (r)=K_1r^{\hat\a}.$$

Now,  we set $\hat u(x):=N^{-1}u(\s x)$ for $\s\in (0,\fr 12)$, where 
$$
N=N(y):= 1\vee \le(2\| u\|_{L^\infty (B_2(y))}\ri)+\fr{2^{\b+1}}{\d_0}\sup_{0<r\leq 2}\fr{1}{r^{\b}}\| f\|_{L^n(B_r(y))}.$$
We shall suppose $y=0$ for simplicity.

It is immediate to see that $\hat u$ is an $L^p$-viscosity subsolution and supersolution, respectively, of
$$
\P^-(D^2u)-\hat\mu (x)|Du|-\hat f(x)=0\quad\mbox{and}\quad 
\P^+(D^2u)+\hat\mu (x)|Du|-\hat f(x)=0\quad\mbox{in }B_2,$$
where $\hat\mu (x):=\s \mu (\s x)$ and $\hat f(x)=\fr{\s^2}{N}f(\s x)$. 
Thus, by Proposition \ref{prop:Holder}, we have
\begin{equation}\label{eq:Holder_hat_u}
|\hat u(x)-\hat u(y)|\leq K_1|x-y|^{\hat\a}\le(\|\hat u\|_{L^\infty (B_2)}+ 
\fr 1N\| f\|_{L^n(B_2)}\ri)\leq K_1|x-y|^{\hat\a}.
\end{equation}
Notice that the last inequality is derived because of our choice of $\d_0$ and $N$.

For $s\in (0,s_0]$, where $s_0:=2^{-\fr 1\b}$, 
we shall find affine functions $\ell_k(x)=
a_k+\la b_k,x\ra$ such that
\begin{equation}\label{approx_affine}
\le\{\begin{array}{ll}
(i)&\| \hat u-\ell_k\|_{L^\infty (B_{2s^k})}\leq s^{k(1+\b)},\\
(ii)&|a_{k-1}-a_k|\vee s^{k-1}|b_{k-1}-b_k|\leq K_2s^{(k-1)(1+\b)},\\
(iii)&|(\hat u-\ell_k)(s^kx)-(\hat u-\ell_k)(s^ky)|\leq K_1s^{k(1+\b)}|x-y|^{\hat\a}
\end{array}
\ri.
\end{equation}
for $x,y\in B_1$, $k\geq 0$, 
where $\ell_{-1}=\ell_0\equiv 0$. 
When $k=0$, it is trivial to check $(i)$ and $(ii)$ while 
$(iii)$ holds by \eqref{eq:Holder_hat_u}.

By induction, assume that \eqref{approx_affine} holds for $k=j$. 
Setting 
$$
v(x):=s^{-j(1+\b)}\le( \hat u( s^j x)-\ell_j(s^j x)\ri),$$
we observe that $v$ is an $L^p$-viscosity solution of
$$
F_j(x,Du,D^2u)+\hat G_j(x,Du,D^2u)-f_j(x)=0\quad\mbox{in }B_2,$$
where 
$$
F_j(x,\xi,X):=\fr{\s^2s^{j(1-\b)}}{N}F\le( \s s^j x,\fr{Ns^{j\b}}{\s}\xi,\fr{N}{\s^2 
s^{j(1-\b)}}X\ri) 
,$$
$$f_j(x):=\fr{\s^2}{N}s^{j(1-\b)}f(\s s^j x),$$
and
$$
\begin{array}{rl}
\hat G_j(x,\xi,X):=&\dis\fr{\s^2s^{j(1-\b)}}{N}\le\{ F\le(\s s^jx, \fr{N}{\s}(s^{j\b}\xi+b_j),\fr{N}{\s^2s^{j(1-\b)}}X\ri)\ri.\\
&\hspace{4cm}\le. \dis -F\le(\s s^j x, \fr{Ns^{j\b}}{\s}\xi,\fr{N}{\s^2s^{j(1-\b)}}X\ri)\ri\}.
\end{array}$$
We note that for $(\xi,X),(\eta,Y)\in\R^n\ti S^n$, 
$$
\P^-(X-Y)-\hat\mu_j (x)|\xi-\eta|\leq F_j(x,\xi,X)-F_j(x,\eta,Y)\leq \P^+(X-Y)+\hat\mu_j(x)
|\xi-\eta|,$$
where $\hat\mu_j(x)=\s s^j\mu (\s s^jx)$. 
Also, since 
$$
|\hat G_j(x,\xi,X)|\leq \s |b_j|s^{j(1-\b)}\mu (\s s^j x)\quad\mbox{for }(x,\xi,X)\in 
\O\ti \R^n\ti S^n,$$
setting $\hat g^*_j(x):=\sup\{ |\hat G_j(x,\xi,X)| \ | \ \xi\in\R^n,X\in S^n\}$, 
we have
\begin{equation}\label{eq:hat_gj}
\| \hat g^*_j\|_{L^n(B_2)}\leq \fr{|b_j|}{s^{j\b}}\| \mu\|_{L^n(B_{2\s s^j})}\leq 
2^{\b_0}\s^{\b_0}|b_j|\o_n\|\mu\|_{L^p(B_{2\s s^j})},
\end{equation}
where $\o_n:=|B_1|^{\fr 1n-\fr 1p}$. 
Hence, 
we immediately verify that 
$v$ is an $L^p$-viscosity subsolution and supersolution, respectively, of
$$
\P^-(D^2u)-\hat\mu_j |Du|-g_j=0 \quad \mbox{and} \quad 
\P^+(D^2u)+\hat\mu_j|Du|+ g_j=0\quad\mbox{
in }B_2,$$
where 
$$
g_j(x):=|f_j(x)|+\hat g^*_j(x).$$
In view of the assumption of our induction, we have
$$
|b_j|\leq  K_2\sum_{k=0}^{j-1}s^{k\b}\leq \fr{K_2}{1-s^\b}\leq 2K_2$$
for $j\geq 1$ because $0<s\leq s_0=2^{-\fr 1\b}$.  
Simple calculations together with our choice of $N$ and \eqref{eq:hat_gj} give
$$
\le\{\begin{array}{ll}
(1)&\| f_j\|_{L^n(B_2)}\leq \fr{\d_0}{2},\\
(2)&\|\hat\mu_j\|_{L^{p'}(B_2)}=(\s s^j)^{1-\fr{n}{p'}}\|\mu\|_{L^{p'}(B_{2\s s^j})}
\leq\s^{1-\fr{n}{p'}}\|\mu\|_{L^{p'}(B_{2 \s s^j})}, \\
(3)&\|\hat g^*_j\|_{L^n(B_2)}
\leq  2^{1+\b_0}K_2
\o_n\s^{\b_0}\|\mu\|_{L^p(B_{2\s s^{j}})}.
\end{array}\ri.
$$
Now, we can choose $\s\in (0,1)$, independent of $j\geq 0$, such that 
$$
\|\hat\mu_j\|_{L^{p'}(B_2)}\vee \| g_j\|_{L^n(B_2)}\leq \d_0. $$
Because $N\geq 1$  and $\s\in (0, 1)$, we verify that 
$$
0\leq \theta_{F_j} (x,y)\leq \theta (x,y),$$
where $\theta_{F_j}(x,y):=\sup_{X\in S^n}| F_j(x,0,X)- F_j(y,0,X)|/(1+\| X\|)$. 

Let $h\in C(\ol B_1)$ be a $C$-viscosity solution of 
$$
F_j(0,0,D^2u)=0\quad\mbox{in }B_1$$
satisfying $h=v$ on $\p B_1$. 
Hence, in view of Lemma \ref{prop:approximate}, we have
\begin{equation}\label{eq:v-h}
\| v-h\|_{L^\infty (B_1)}\leq \e=K_2s^{1+\hat \b}.
\end{equation}

We  define
$$
\ell_{j+1}(x):=\ell_j(x)+s^{j(1+\b)}\le( h(0)+\la Dh(0),\fr{x}{s^j}\ra\ri).$$
Since we observe that for $|x|\leq 1$, by \eqref{eq:v-h} and the fact $h\in C^{1+\hat \b}(\ol B_{\fr 12})$, we have
$$
|\hat u(2s^{j+1}x)-\ell_{j+1}(2s^{j+1}x)|
\leq s^{(j+1)(1+\b)} s^{\hat \b-\b}\le( K_2+2^{1+\hat \b}K_2\ri)
$$
for $s\in (0,s_1]$, where $s_1:=s_0\wedge (K_2+2^{1+\hat \b}K_2)^{-\fr{1}{\hat \b-\b}})$ by \eqref{eq:v-h}, 
$(i)$ holds for $k=j+1$. 

To show $(ii)$ for $k=j+1$, by Proposition \ref{prop:C1beta}, we first verify 
$$
\| h\|_{C^1(\ol B_{\fr 12})}\leq K_2\| h\|_{L^\infty (B_1)}=K_2\max_{\p B_1}|v|\leq K_2.$$
Thus, noting $|b_{j+1}-b_j|=s^{j\b}|Dh(0)|$ and $a_{j+1}-a_j=s^{j(1+\b)}h(0)$, we 
obtain $(ii)$ for $k=j+1$. 

In order to see $(iii)$ for $k=j+1$, setting 
$$
\hat v(x):=v(x)-h(0)-\la Dh(0),x\ra,$$
we observe that for $x\in B_1$, 
$$
\begin{array}{rl}
|\hat v(2sx)|\leq& |v( 2sx)-h(2sx)|+|h(2sx)-h(0)-\la Dh(0),2sx\ra|\\
\leq&2K_2s^{1+\hat\b}\\
= &2s^{1+\b}s^{\hat \b -\b}K_2.
\end{array}$$
Thus, for $s\in (0,s_2]$, where $s_2:=1/( 8K_2)^{\fr{1}{\hat \b-\b}}$, 
we have
$$
\| \hat v\|_{L^\infty (B_{2s})}\leq \fr{s^{1+\b}}{4}.$$

We next verify that $\hat v$ is an $L^p$-viscosity solution of
$$
F_j(x,Du,D^2u)+G_j(x,Du,D^2u)-f_j=0\quad\mbox{in }B_2,$$
where 
$$
\begin{array}{rl}
 G_j(x,\xi,X):=&\dis\fr{\s^2s^{j(1-\b)}}{N}\le\{ F\le(\s s^jx, \fr{N}{\s}(s^{j\b}\xi+b_j+s^{j\b}Dh(0)),\fr{N}{\s^2s^{j(1-\b)}}X\ri)\ri.\\
&\hspace{4cm}\le. \dis -F\le(\s s^j x, \fr{Ns^{j\b}}{\s}\xi,\fr{N}{\s^2s^{j(1-\b)}}X\ri)\ri\}.
\end{array}$$
Hence, as before
we observe that 
$$
|G_j(x,\xi,X)|\leq \s |b_j+s^{j\b}Dh(0)|s^{j(1-\b)}\mu (\s s^j x)=:h_j(x)\quad\mbox{for }(x,\xi,X)\in 
\O\ti \R^n\ti S^n.$$
Since  
we have
$$
\| f_j\|_{L^n(B_{2s})}\leq \fr{s^\b}{2},
$$
and
$$
\begin{array}{rcl}
\| h_j\|_{L^n(B_{2s})}\leq
 \dis\fr{3K_2}{s^{j\b}}\|\mu\|_{L^n(B_{2\s s^{j+1}})}
 &\leq &6K_2 \o_n s^{(j+1)\b_0-j\b}\|\mu\|_{L^p(B_{2\s s^{j+1}})}\\
&\leq& 6K_2\o_n s^{\b_0}\|\mu\|_{L^p(B_{2\s s^{j+1}})},
\end{array}
$$
we see that $x,y\in B_s$, 
$$
|\hat v(x)-\hat v(y)|\leq K_1s^{1+\b-\hat\a}|x-y|^{\hat\a}\le(\fr 34 +6s^{\b_0-\b}
K_2\o_n\|\mu\|_{L^p(B_{2s})}\ri) .$$
Hence, we can choose smaller $s>0$, if necessary, to obtain that
$$
|\hat v(x)-\hat v(y)|\leq  K_1s^{1+\b-\hat\a}|x-y|^{\hat\a}\quad\mbox{for }x,y\in B_s.$$
Now, for $x,y\in B_1$, we calculate in the following way:
$$
\begin{array}{rl}
&|(\hat u-\ell_{j+1})(s^{j+1}x)-(\hat u-\ell_{j+1})(s^{j+1}y)|\\
\leq&|(\hat u-\ell_j)(s^{j+1}x)-(\hat u-\ell_j)(s^{j+1}y)-
s^{j(1+\b)}\la Dh(0),s(x-y)\ra|\\
=&s^{j(1+\b)}|\hat v(sx)-\hat v(sy)|\\
\leq&K_1s^{(j+1)(1+\b)}|x-y|^{\hat\a}.
\end{array}
$$

Thanks to $(ii)$ of \eqref{approx_affine}, we find $a_\infty\in\R$ and $b_\infty\in\R^n$ such that 
$(a_k,b_k)\to (a_\infty,b_\infty)$ as $k\to\infty$. 
For any $x\in B_1$, we choose $k\in \N$ such that 
$$
2s^{k+1}\leq |x|<2s^k.$$
Since  $(i)$ of \eqref{approx_affine} yields
$$|\hat u(x)-a_k-\la b_k,x\ra|\leq s^{k(1+\b)}\leq \fr{1}{(2s)^{1+\b}}|x|^{1+\b},$$
by sending $k\to\infty$, it follows
$$
|\hat u(x)-a_\infty -\la b_\infty ,x\ra|\leq \fr{1}{(2s)^{1+\b}}|x|^{1+\b}.$$
Therefore, it is standard to establish the H\"older continuity of $Du$ with its exponent $\b$. 
See \cite{Krylov} or \cite{AP} for instance. 
\end{proof}


\subsection{Estimates near the coincidence set}

We next prove that the first derivative of $L^p$-viscosity solutions $u$ of \eqref{eq:1}
is H\"older continuous with exponent $\b_0\wedge \b_1$ near  the coincidence set $C^\pm [u]$, 
where $u$ touches  one of the obstacles.

In what follows, for the $L^p$-viscosity solution $u\in C(\O)$ of 
\eqref{eq:1}, we use the  notation of $\e$-neighborhood of $C^\pm [u]$ 
 for small $\e>0$;
$$
C^\pm_\e [u]:=\{ x\in \O \ | \ \mbox{dist}(x,C^\pm[u])<\e\}.
$$


\begin{lem}\label{lem:C1Holder}
Assume \eqref{Apq2}, \eqref{Af}, \eqref{Amu}, \eqref{UE}, \eqref{Azero},  \eqref{Aobstacles2} and \eqref{Aregob2}. 
Then, for small $\e>0$, there exists $\hat C_0=\hat C_0(\e)>0$  such that 
if $u\in C(\O)$ is an $L^p$-viscosity solution of \eqref{eq:1}, and 
$x_0\in C^-[u]\cap \O_\e$ (resp., $C^+[u]\cap \O_\e$), then it follows that 
$$
|u(x)-u(x_0)-\la D\phi(x_0),x-x_0\ra |\leq \hat C_0r^{1+\b_2}$$
$$
\le( \mbox{resp., } |u(x)-u(x_0)-\la D\psi (x_0),x-x_0\ra |\leq \hat C_0r^{1+\b_2}\ri)
$$
for $x\in B_r (x_0)$. 
In particular, $u$ is differentiable at $x_0$, and 
$$Du(x_0)=D\phi (x_0)\quad (\mbox{resp., }Du(x_0)=D\psi (x_0)).$$ 
\end{lem}

\begin{proof}
We consider the case when $x_0\in C^-[u]\cap \O_\e$; $(u-\phi)(x_0)=0$. 
For simplicity of notations, we shall suppose $x_0=0\in C^-[u]\cap\O_\e$. 

Because of \eqref{Aobstacles2}, we choose  small $r>0$ such that
$$
u(x)<\psi (x)\quad \mbox{in }B_{4r}.$$
Hence, setting  $v(x):=u(x)-\phi (0)-\la D\phi (0),x\ra+Ar^{1+\b_1}$ for 
a large $A>0$, 
we observe that   $v$ is a nonnegative $L^p$-viscosity supersolution of 
$$
\P^+(D^2u)+\mu |Du|+f^-+|D\phi (0)|\mu=0\quad\mbox{in }B_{4r}.
$$
In view of Proposition \ref{prop:WH}, there is $\e_0>0$ such that 
$$
r^{-\fr n{\e_0}}\| v\|_{L^{\e_0}(B_{2r})}\leq C\le(\inf_{B_{2r}}v+r^{1+\b_0}\| f^-+\mu \|_{L^n(B_{4r})}\ri)\leq C\le(r^{1+\b_1}+r^{1+\b_0}\ri).
$$
Thus, from our choice of  $\b_2$, we have
\begin{equation}\label{ineq:wH}
r^{-\fr n{\e_0}}\| v\|_{L^\e (B_{2r})}\leq Cr^{1+\b_2}.
\end{equation}

On the other hand, we claim  that 
$w:=v\vee A'r^{1+\b_1}$, where $A'>A$,  is also an $L^p$-viscosity subsolution  of
$$
\P^-(D^2u)-\mu |Du|-f^+-|D\phi (0)|\mu=0\quad\mbox{in }B_{4r}.$$
Indeed, assuming that $w-\xi$ attains its local maximum at 
$z\in B_{4r}$ for $\xi\in W^{2,p}(B_{4r})$, 
we shall conclude the claim. 
In case of  $w(z)=v(z)$,  noting 
$$
u>\phi\quad \mbox{near }z$$
for large $A'>A$, we observe that 
$w$ is an $L^p$-viscosity subsolution of
\begin{equation}\label{eq:sub}
\P^-(D^2u)-\mu |Du|-f^+-|D\phi (0)|\mu=0
\end{equation}
in $B_{\hat r}(z)$ 
for some $\hat r>0$ while in case of $w(z)=A'r^{1+\b_1}$, we immediately see that 
 any constant is an $L^p$-viscosity subsolution of \eqref{eq:sub}.  
 Hence, we verify that $w$ is an $L^p$-viscosity subsolution of \eqref{eq:sub} in 
 $B_{4r}$. 

In view of Proposition \ref{prop:LMP}, 
 with the above $\e_0>0$, we have
$$
\sup_{B_{r}}v\leq C_1\le( r^{-\fr n{\e_0}}\| w\|_{L^{\e_0}(B_{2r})}+r^{1+\b_0}\| f^++\mu\|_{L^p(B_{4r})}\ri),
$$
where $C_1=C_1(\e_0)$ is the constant in Proposition \ref{prop:LMP}. 
This together with  \eqref{ineq:wH} implies 
$$
-Cr^{1+\b_1}\leq u(x)-\phi (0)-\la D\phi (0),x\ra 
\leq Cr^{1+\b_2}\quad \mbox{in }B_{r},
$$
which concludes the proof. 
\end{proof}

Thanks to Lemma \ref{lem:C1Holder} with Proposition \ref{prop:C1Holder}, 
we easily obtain Theorem 2. 12. 
We give a brief proof though it seems standard. 


\begin{proof}[Proof of Theorem \ref{thm:C1Holder1}]
In view of Proposition 5.1, to complete the assertion, 
we may suppose $x,y\in  C^\pm_{2r_1}[u]\cap \O_{2r_1}$.  
Furthermore, by Lemma \ref{lem:C1Holder}, we may  suppose that
$x,y\in C^-_{2r_1}[u]$, and $0<\mbox{dist}(y,C^-[u])\leq 
\mbox{dist}(x,C^-[u])$. 
Choose $\hat x,\hat y\in C^-[u]$ such that $|x-\hat x|=\mbox{dist}(x,C^-[u])$ and $|y-\hat y|=\mbox{dist}(y,C^-[u])$. 
Thus, we have
$$0<|y-\hat y|\leq |x-\hat x| .$$

\ul{Case 1: $|x-y|<\fr 12 |x-\hat x|$.}   
In view of Proposition \ref{prop:C1Holder},  for any $\b\in (0,\b_0\wedge \hat\b)$,  
we easily obtain
$$
|Du(x)-Du(y)|\leq C|x-y|^{\b}.$$

\ul{Case 2: $|x-y|\geq \fr 12 |x-\hat x|\ge \fr 12 |y-\hat y|$.}   
In view of Lemma \ref{lem:C1Holder}, we have 
$$
\begin{array}{rcl}
|Du(x)-Du(y)|&\leq &|Du(x)-Du(\hat x)|+|Du(\hat x)-Du(\hat y)|\\
&&+|Du(\hat y)-Du(y)|\\
&\leq &C(|x-\hat x|^{\b_2}+|y-\hat y|^{\b_2})+C|\hat x-\hat y|^{\b_1},
\end{array}
$$
which is estimated by $C|x-y|^{\b_2}$ in this case.  

Therefore, combining these cases, we obtain the desired estimate. \end{proof}

\end{document}